\documentclass[12pt]{amsart}
\usepackage[english]{babel}
\usepackage{egothic}
\usepackage[T1]{fontenc}

\newcommand{\longhookleftarrow}
{\ensuremath{\joinrel\relbar\joinrel\relbar\joinrel\leftarrow\rhook}}
\usepackage{amsmath}
\usepackage{amsthm}
\usepackage{amssymb}
\usepackage{mathrsfs}
\usepackage{enumerate}
\usepackage{physics}
\usepackage[notcite, final, notref]{showkeys}
\usepackage{dsfont}
\usepackage[dvips]{color}
\newtheorem{theorem}{Theorem}[section]

\newtheorem{lemma}[theorem]{Lemma}
\newtheorem{proposition}[theorem]{Proposition}
\newtheorem{corollary}[theorem]{Corollary}
\newtheorem{definition}[theorem]{Definition}

\theoremstyle{definition}
\newtheorem{remark}[theorem]{Remark}

\newtheorem{example}[theorem]{Example}

\usepackage{xcolor}

\title[New characterizations of reproducing kernel Hilbert spaces]{New characterizations of reproducing kernel Hilbert spaces  and applications to metric geometry}

\author[D. Alpay]{Daniel Alpay}
\address{(DA) Schmid College of Science and Technology \\
Chapman University\\
One University Drive
Orange, California 92866\\
USA}
\email{alpay@chapman.edu}

\author[P. Jorgensen]{Palle Jorgensen}
\address{(PJ)
Department of Mathematics, 14 MLH\\The University
of Iowa, Iowa City, Iowa 52242-1419\\ USA}
\email{jorgen@math.uiowa.edu}

\begin{document}
\maketitle
\begin{abstract}
  We give two new global and algorithmic constructions of the reproducing kernel Hilbert space associated to a positive definite kernel. We further present a
  general positive definite kernel setting using
  bilinear forms, and we provide new examples. Our results cover the case of measurable positive definite kernels, and we give applications to both stochastic analysis
  and metric geometry and provide a number of examples.
\end{abstract}

\noindent AMS Classification. 46E22, 43A35

\noindent Keywords: reproducing kernel, positive definite functions, approximation, algorithms, measures, stochastic processes.

\section{Introduction}
\setcounter{equation}{0}

Recently many researchers have made use of reproducing kernels in attacking diverse areas and problems from approximation and optimization; see 
\cite{MR3546121,MR3683690,MR1254059,MR1637511,MR4033493,MR3940262}. While there is a direct link from specific positive definite kernels, the link to the corresponding 
reproducing kernel Hilbert space (RKHS) is a rather abstract one, typically, one is faced only with an abstract completion, and the links to computation is often blurred. Our present aim is to offer two concrete approaches to RKHSs, explicit, and algorithmic. Our approaches are dictated directly by problems in constructive approximation and 
optimization.\smallskip

Let $S$ be some set and let $k(t,s)$ be positive definite on $S$. It is well known (see \cite{aron,meschkowski,saitoh}) that one can
associate to $k$ a uniquely defined Hilbert space (which we will denote by $\mathfrak H(k)$)
of functions defined on $S$ with reproducing kernel $k(t,s)$, meaning that for every $s\in S$ and $f\in\mathfrak H(k)$, the function $t\mapsto k(t,s)$ belong to $\mathfrak H(k)$ and
\begin{equation}
\langle f(\cdot),k(\cdot, s)\rangle_{\mathfrak H(k)}=f(s).
\label{repro-ker}
\end{equation}
The classical construction of the space $\mathfrak H(k)$ goes as follows. Consider the linear span $\mathfrak L(k)$
of the functions $k_s\,:\, u\mapsto k(u,s)$ endowed with the form
\begin{equation}
\label{number1.1}
\langle k_s,k_t\rangle_k=k(t,s),\quad t,s\in S.
\end{equation}
Then, one readily proves that the form \eqref{number1.1} is well defined, and has the reproducing kernel property. Thus $(\mathfrak L(k),\langle\cdot\, ,\,\cdot\rangle_k)$ is a pre-Hilbert space. Its metric space completion is unique up to a metric space isometry. One then shows that a uniquely defined completion exists,
which is made of functions, and which still has the reproducing kernel property.\smallskip

In general there is no direct characterization of the elements of $\mathfrak H(k)$ directly as functions on $S$ (notable exceptions include for instance the Fock space and the Bergman space).
In the present work we give two new global constructions of $\mathfrak H(k)$ and a global characterization of its elements. The first construction is in terms of
a projective limit, and the second one uses a space of measures associated to the kernel. In the latter, we associate an explicit norm, for every positive kernel, and valid for every element in the reproducing kernel Hilbert space; see formula
\eqref{norms}.\smallskip

Another space plays an important role in our constructions besides the space $\mathfrak L(k)$, namely the linear span $\mathfrak E(k)$
of the Dirac measures $\delta_s,s\in S$, endowed with the Hermitian form
defined first by
\begin{equation}
\label{deltatdeltas}  
\langle \delta_s,\delta_t\rangle=k(t,s),\quad t,s\in S.
\end{equation}
and extended by linearity to the finite linear combinations of Dirac measures.
In this case,  \eqref{deltatdeltas}  generates a positive, but possibly degenerate, form, and one needs to mod-out via the linear space of finite linear combinations of
Dirac measures which are self-orthogonal with respect to this Hermitian form.\\

We now describe the outline of the paper. In Section 2, starting from a fixed positive definite function (kernel) $k(t,s)$ defined on $S\times S$,
we derive an associated metric $d_k$ on $S$. Its properties are outlined; as well as its applications inside the paper. Section 3 deals with a duality lemma important in
the main constructions inside the paper: The starting point here is a general pre-Hilbert space, so a normed vector space where the norm is defined from an inner product.
Hence, we get a corresponding dual norm, i.e., a dual Banach space. With the use of a new transform, we show that this dual Banach space is a Hilbert space.
In Section 4,  starting from a fixed positive definite function (kernel) $k(t,s)$ defined on $S\times S$ , we present a constructive approach to the reproducing kernel
Hilbert space (RKHS) $\mathfrak H(k)$. This will make use of the following five steps: (i) the filter of all finite subsets of S, (ii) an explicit system of finite-rank
operators defined on functions on S, and (iii) an algorithm which related the operators defined from two finite sets, with one contained in the other. With this,
we establish (iv) a Kolmogorov consistency relation, and (v) we show that the constructive and algorithmic realization of the RKHS $\mathfrak H(k)$ follows directly
from this (Theorem \ref{4.2}). It is helpful to draw the following parallel between our particular choice of filter (i.e., the filter of all finite subsets) which deals
with the most general setting, and other hand, related filters for special cases, in the literature. In more detail, on the one hand, we have (i) for a given positive
definite kernel function $k(t,s)$ on a set $S$, our present filter consists of all finite subsets of $S$. By contrast, (ii) for special RKHS constructions, one has a
variety of alternative choices filters. They are used in special cases of RKHS constructions from harmonic analysis; for example, for Hardy spaces $\mathbf H_2(D)$, with
$D$ some complex domain,  it is popular to pick a filter consisting of closed curves inside $D$. There are interesting similarities and differences between
(ii) special cases, special filters, and (i) our case: the general case, and with our choice of filter.
Motivated by applications, in Section 5, we introduce a pre-Hilbert space of measures on $S$, with the measures defined relative to the
Borel sigma-algebra from the metric $d_k$. Our duality lemma from Section 3 will be used in a characterization of the RKHS norm in $\mathfrak H(k)$, Theorem \ref{thm1}.
The proof of this result is of independent interest, and is carried out in detail, in Section 6. Sections 8 and 9 in turn deal with a RKHS construction for a class of
measurable positive definite kernels. Finally, Section 10 contains applications to stochastic processes and Section 11 contains applications to geometric measure theory.

\section{The $d_k$-metric}
\setcounter{equation}{0}
\label{dist}

Our present starting point is a fixed set $S$ and a positive definite function (kernel) $k(t,s)$ defined on $S\times S$. From this we derive an associated metric $d_k$ on
$S$. We shall use the name ``metric'' even for the cases when $d_k$ might in fact not separate the points in $S$. It will be important for our subsequent construction to
have $d_k$ derived from first principles. For example, we shall specify our Borel sigma-algebra of subsets in $S$ with use of the metric. This will
be important in our constructive realization of the reproducing kernel Hilbert space (RKHS) $\mathfrak H(k)$ associated with the given kernel $k$. (Naturally, once
$\mathfrak H(k)$ is available, then the metric $d_k$  will be an easy byproduct. But our point in the present section is to derive the metric $d_k$ directly from an easy
matrix consideration.) For our present purpose, the metric $d_k$  will be available at the outset, and we show in Proposition \ref{procont} that the kernel function $k$ is
automatically jointly continuous relative to the product metric.\smallskip



The following is well known, and will hold even when $k(t,s)$ is only Hermitian. A proof is outlined for completeness.

\begin{lemma}
Let $S$ be a set and let $k(t,s)$ be positive definite on $S$. The formula \eqref{number1.1}
induces a well defined non-degenerate Hermitian form on the linear span $\mathfrak L(k)$ of the functions $k_s$, $s\in S$, and the space $\mathfrak L(k)$ endowed with this form is
a reproducing kernel pre-Hilbert space.
\end{lemma}

\begin{proof}
The form \eqref{number1.1} extended to the linear span of the kernel functions, is well defined, and the reproducing kernel property stems directly from \eqref{number1.1}.
We now prove that the form is non-degenerate. Let $f\in\mathfrak L(k)$ be orthogonal to all the elements in the space. We have in particular
\[
\langle f,k_s\rangle_k=0,\quad \forall s\in S
\]
and so, using the reproducing kernel property, $f(s)=0$ for all $s\in S$.
\end{proof}

The norm associated to this form induces a distance via
\[
d(f,g)=\|f-g\|_k,\quad f,g\in\mathfrak L(k).
\]
In particular
\[
d(k_t,k_s)^2=k(t,t)-2{\rm Re}\, k(t,s)+k(s,s),\quad t,s\in S.
\]  
The right-hand side of the above equality is non negative since $k$ is positive definite. We set
\begin{equation}
d_k(t,s)=d(k_t,k_s),\quad t,s\in S.
\label{dkdistance}
\end{equation}

We note that $d_k$ is symmetric and satisfies the triangle inequality since, for $t,s,u\in S$ we have:
\[
\begin{split}  
d_k(t,s)&=\|k_t-k_s\|_k\\
&\le\|k_t-k_u\|_k+\|k_u-k_s\|_k\\
&=d_k(t,u)+d_k(u,s).
\end{split}
\]  

It need not define a distance $d_k(t,s)$ on $S$, since it may happen that $d_k(t_0,s_0)=0$ with $t_0\not= s_0$.
We then have:
\[
k(s,t_0)=k(s,s_0),\quad \forall s\in S.
\]

We will say that two points $t_0$ and $s_0$ are equivalent if $d(k_{t_0},k_{s_0})=0$. We have an equivalence relation $\sim$, and we consider the function $k(t,s)$ on the quotient
space $S/\sim$. We set
\begin{equation}
\widetilde{k}(\widetilde{t},\widetilde{s})=k(t,s),\quad t\in\widetilde{t},\,\,\, s\in\widetilde{s}
\end{equation}
and
\begin{equation}
\widetilde{d}_k(\widetilde{t},\widetilde{s})=d_k(t,s),\quad t\in\widetilde{t},\,\,\, s\in\widetilde{s}.
\end{equation}  

The proofs of the following two results are easy and omitted.
\begin{lemma}
$\widetilde{k}$ is well defined and positive definite on $S/\sim$.
\end{lemma}  


\begin{lemma} $\widetilde{d}_k$ is well defined and $(S/\sim,\widetilde{d}_k)$ is a metric space.
\end{lemma}  



In the sequel we assume, without loss of generality, that $d_k$ is indeed a distance on $S$ (i.e. we consider $S/\sim$ rather than $S$).

\begin{proposition}
The function $k(t,s)$ on $S\times S$ is jointly continuous with respect to the product metric $D_k((t_1,s_1),(t_2,s_2))=d_k(t_1,t_2)+d_k(s_1,s_2)$ and
the diagonal function $t\mapsto k(t,t)$ is continuous with respect to $d_k$.
\label{procont}
\end{proposition}

\begin{proof} We first consider the second claim.
The triangle inequality in $\mathfrak L(k)$ gives
\[
\left|\|k_t\|_k-\|k_t\|_k\right|\le \|k_t-k_s\|_k,
\]
which can be rewritten as
\[
\left| \sqrt{k(t,t)}-\sqrt{k(s,s)}\right|\le d_k(t,s).
\]
Thus the map $t\mapsto \sqrt{k(t,t)}$ is continuous with respect to $d_k$ and so is its square.\\

To prove the joint continuity of $k(t,s)$ with respect to $D_k$ we write:
\[
\begin{split}
  \left|k(t_1,s_1)-k(t_2,s_2)\right|&\le \left|k(t_1,s_1)-k(t_1,s_2)\right|+ \left|k(t_1,s_2)-k(t_2,s_2)\right|\\
  &=|\langle k_{s_1}-k_{s_2}, k_{t_1}\rangle|+|\langle k_{s_2}, k_{t_1}-k_{t_2}\rangle|\\
&\le \sqrt{k(t_1,t_1)}d_k(s_1,s_2)+\sqrt{k(s_2,s_2)}d_k(t_1,t_2).
\end{split}    
\]  
  
\end{proof}  

\begin{example}
  For the covariance of the Brownian motion, $k(t,s)=t\wedge s$ with $t,s\ge 0$ we have
  \[
    d_k(t,s)=\sqrt{|t-s|}.
    \]
\end{example}  

We remark that:

\begin{proposition}
The elements of $\mathfrak L(k)$ are Lipschitz continuous with respect to $d_k$
\end{proposition}  

\begin{proof}
  The claim follows from the Cauchy-Schwarz inequality. Indeed, let $f\in\mathfrak L(k)$ and $t,s\in S$. Then
  \[
|f(t)-f(s)|=\langle f,k_t-k_s\rangle_k\le \|f\|_k\cdot\|k_t-k_s\|_k=\|f\|_k\cdot d_k(t,s).
  \]  
\end{proof}

The metric $d_k$ induces an associated topology and sigma-algebra, denoted by $\mathcal T_k$ and $\mathcal B_k$ respectively. We complete $\mathcal B_k$ with the sets of
outer measure $0$, and still denote the completion $\mathcal B_k$.

\section{A canonical isometry operator}
\setcounter{equation}{0}
\label{sec2}
Let $(\mathfrak H_0,\langle\cdot, \cdot\rangle_{\mathfrak H_0})$ denote a pre-Hilbert space. We take the inner product linear in the first variable and anti-linear in the second variable. Furthermore, we
denote by $(\mathfrak H,\langle\cdot,\cdot\rangle_{\mathfrak H})$ any of its completion. Recall that it is unique up to a Hilbert space isomorphism.
Let $\mathfrak H_0^*$ denote the Banach anti-dual of $\mathfrak H_0$, that is the space of anti-linear continuous maps from $\mathfrak H_0$ into $\mathbb C$.

\begin{lemma}
The dual norm defined by
\begin{equation}
  \label{poi}
\|\varphi\|_{\mathfrak H_0^*}  =\sup_{\substack{h\in\mathfrak H_0\\  \|h\|_{\mathfrak H_0}=1}}|\varphi(h)|,\quad\varphi\in \mathfrak H_0^*,
\end{equation}
is a Hilbert norm, and $\mathfrak H_0^*$ is complete with this norm.
\end{lemma}

\begin{proof}
It follows from \eqref{poi} that $\varphi$ is also continuous from the chosen closure $\mathfrak H$ of $\mathfrak H_0$ into $\mathbb C$.
By Riesz theorem applied to $\varphi$ as a continuous map from $\mathfrak H$
into $\mathbb C$, there exists an element $T\varphi\in\mathfrak H$ such that
\begin{equation}
\label{denseT}
\varphi(h)=\langle T\varphi, h\rangle_{\mathfrak H}.
\end{equation}  
It follows that 
\[
\|\varphi\|_{\mathfrak H_0^*}=\|T\varphi\|_{\mathfrak H},
\]
and so the norm \eqref{poi} is defined by the inner product:
\begin{equation}
  \langle\varphi,\psi\rangle_{\mathfrak H_0^*}\stackrel{\rm def.}{=}\langle T\varphi,T\psi\rangle_{\mathfrak H}.
\label{33}
\end{equation}
\end{proof}

\begin{lemma}
In the notation of the previous lemma and of its proof, the operator $T$ is a Hilbert space isomorphism from
$\mathfrak H_0^*$ onto $\mathfrak H$.
\end{lemma}

\begin{proof}
By \eqref{33}, $T$ is an isometry and hence is enough to show that $T$ has dense range. This follows from \eqref{denseT}.
\end{proof}

The function
\begin{equation}
  \label{rk12}
K(\varphi,\psi)=\langle T\varphi,T\psi\rangle_{\mathfrak H},\quad \varphi,\psi\in\mathfrak H_0^*,
\end{equation}
is positive definite, and independent of the chosen completion of $\mathfrak H_0$.
We associate to it the Hilbert space of functions of the form
\begin{equation}
F(\varphi)=\langle T\varphi, f\rangle_{\mathfrak H},\quad \|F\|=\|f\|_{\mathfrak H}
\end{equation}

In the case of the point evaluations we have
\[
  \begin{split}
    \overline{f(t)}&=\overline{\delta_t}(f)\\
    &=\langle k(\cdot, t),f\rangle_{\mathfrak H}
  \end{split}
\]
and so
\[
T(\delta_t)=k(\cdot, t)
\]
and the reproducing kernel \eqref{rk12} restricted to the point measures is (with $\mathfrak H_0$ the linear span of the $\delta_t$ with the form \eqref{number1.1})
\begin{equation}
K(\delta_t,\delta_s)=\langle k(\cdot, t),k(\cdot, s)\rangle_{\mathfrak H}=k(s,t).
  \end{equation}
Not always are point evaluations in $\mathfrak H_0^*$ (in particular $\mathfrak H_0$ need not be a space of functions) but the kernel \eqref{rk12} is associated in a natural way to
$\mathfrak H_0$.
\section{An injective limit construction}
\setcounter{equation}{0}
\label{secF}
The idea is to look at the filter $\mathcal F$ of finite subsets of $S$. For any such set $F$, one considers the finite dimensional reproducing kernel Hilbert space 
$\mathfrak L_F(k)$ with reproducing kernel $k(t,s)\big|_{F\times F}$. The argument below uses a limit over the filter of all finite subsets.
We denote by $K_F$ the corresponding Gram matrix. In order to get only invertible matrices we first extract from $S$ a possibly smaller set $S_0$ for which the functions
$k_t, t\in S_0$ are linearly independent and span the linear pan of the $k_t,t\in S$.
We then define the projection from the vector space of functions defined on $F=\left\{s_1,\ldots, s_m\right\}\subset S_0$
onto $\mathfrak L_F(k)$ by
\begin{equation}
  \label{interpo}
  (Q_Ff)(s)=\begin{pmatrix}k(s,s_1)&\cdots& k(s,s_m)\end{pmatrix}K_F^{-1}
  \underbrace{\begin{pmatrix}f(s_1)\\ \vdots\\ f(s_m)\end{pmatrix}}_{\stackrel{\rm def.}{=}f|_F},
\end{equation}
where the variable $s$ runs through $F$. To ease the notation, we write $S$ instead of $S_0$ in the sequel.

\begin{lemma}
  The map $Q_F$ has the following properties:\\
$(1)$ Interpolation:  
\begin{equation}
  \label{interpo2}
  (Q_Ff)(s_j)=f(s_j),\quad j=1,\ldots, m.
\end{equation}
$(2)$ Projection:
  \[
    Q_F^2=Q_F.\]
$(3)$ Norm:
\begin{equation}
  \label{norm2}
\|Q_Ff\|^2_{\mathfrak L_F(k)}=(f|_F)^*K_F^{-1}(f|_F).
\end{equation}
$(4)$ Order preserving: if $F_1\subset F_2$ are finite subsets of $S$,
\[
Q_{F_2}Q_{F_1}=  Q_{F_1}Q_{F_2}=Q_{F_1}.
\]
\end{lemma}
\begin{proof}
  Setting $s=s_i$ on the right-hand side of \eqref{interpo} we note that $\begin{pmatrix}k(s_i,s_1)&\cdots& k(s_i,s_m)\end{pmatrix}$ is the $i$-th row of $K_F$, and so
  $  \begin{pmatrix}k(s_i,s_1)&\cdots& k(s_i,s_m)\end{pmatrix}K_F^{-1}$ is the $1\times m$ row vector with all entries equal to $0$ besides the $i$-th entry equal to $1$. Hence \eqref{interpo2} holds.
  Applying the preceding argument to $Q_Ff$ we get $Q_F^2=Q_F$. \eqref{norm2} follows from
  \[
    \begin{split}
      \langle Q_Ff,Q_Ff\rangle_{\mathfrak L(k)}&=\langle \sum_{k,j=1}^m k(\cdot, s_j) (K_F^{-1})_{jk}\overline{f(s_k)},\sum_{i,\ell=1}^mk(\cdot, s_i) (K_F^{-1})_{i\ell}
      \overline{f(s_\ell)}(K_F^{-1})_{\ell i}(K_F)_{ij}
\rangle_{\mathfrak L(k)}\\
&=\sum_{j,k,i,\ell=1}^mf(s_\ell)(K_F^{-1})_{\ell i}(K_F)_{ij}(K_F^{-1})_{jk}\overline{f(s_k)}\\
&=\sum_{j,k,\ell=1}^mf(s_\ell)\left(\sum_{i=1}^m(K_F^{-1})_{\ell i}(K_F)_{ij}\right)(K_F^{-1})_{jk}\overline{f(s_k)}\\
&=\sum_{j,k,\ell=1}^mf(s_\ell)\delta_{\ell j}(K_F^{-1})_{jk}\overline{f(s_k)} \\
&=\sum_{j,k,\ell=1}^mf(s_\ell)\delta_{\ell j}(K_F^{-1})_{jk}\overline{f(s_k)}\\
  &=\sum_{j,k=1}^mf(s_j)(K_F^{-1})_{jk}\overline{f(s_k)},
\end{split}
\]  
which is \eqref{norm2}.\smallskip

Finally, let $F_1$ and $F_2$ be two finite subsets of $S$ such that $F_1\subset F_2$. Let $f$ be a function defined on $S$. The functions $f$ and $Q_{F_2}f$ coincide on $F_2$ and so on $F_1$, and so
\[
Q_{F_1}f=Q_{F_1}(Q_{F_2}f)
\]
and so $Q_{F_1}=Q_{F_1}Q_{F_2}$.
Next, $Q_{F_2}(Q_{F_1}f)=Q_{F_1}$ since the span of the kernels for $s\in F_1$ is isometrically included in the span of the kernels for $s\in F_2$.
\end{proof}

As a definition we introduce the  subspace of all the functions on $S$ for which the sup
\[
  \sup_{F\in\mathcal F}\|Q_Ff\|^2_{\mathfrak L_F(k)}
  \]
is finite. In the next theorem we show that this is the reproducing kernel Hilbert space associated to $k$.

\begin{theorem}
  Let $k(t,s)$ be positive definite on $S$.\\
  $(1)$ The following injective limit of Hilbert spaces
  \begin{equation}
    \label{space???}
\lim_{\substack{F\in\mathcal F\\ K_F>0}}\mathfrak L_F(k)
\end{equation}
as $F\rightarrow S$, with norm
\begin{equation}
  \label{norm?}
\|f\|^2_{\mathfrak H}=\sup_{F\in\mathcal F}\|Q_Ff\|^2_{\mathfrak L_F(k)}
\end{equation}
exists.\\
$(2)$ The supremum in \eqref{norm?} defines a Hilbert space norm, and the above injective limit coincides with the reproducing kernel Hilbert space
with reproducing kernel $k(t,s)$.
\label{4.2}
\end{theorem}

\begin{proof} We first note that when $F_1\subset F_2$, the space $\mathfrak L_{F_1}(k)$ is isometrically included in the space $\mathfrak L_{F_2}(k)$.
We denote by $J=J_{F_1F_2}$ the isometric inclusion from
$\mathfrak L_{F_1}(k)$ into $\mathfrak L_{F_2}(k)$.  The operator $JJ^*=Q_{F_1F_2}$ is the orthogonal projection from $\mathfrak L_{F_2}(k)$ onto
$\mathfrak L_{F_1}(k)$. By Kolmogorov's consistency theorem there is a uniquely defined Hilbert space $\mathfrak H$ such that

\[\renewcommand{\arraystretch}{1.5}
\begin{array}{ccccc}
\mathfrak H&&\stackrel{i}{\longhookleftarrow}
 & &\mathfrak L_{F_2}(k)\\
&\stackrel{i}{\nwarrow}& &\stackrel{J}{\nearrow}\\
&  &\mathfrak L_{F_1}(k)
\end{array}.
\]

To conclude the proof we now proceed in a number of steps.\\
  
  STEP 1: {\sl \eqref{norm?} defines a norm:} The arguments for a norm are easily checked, but it is important to note that the supremum is in fact an increasing limit since
  \[
F_1\subset F_2\,\,\Longrightarrow\,\,\|Q_{F_1}f\|^2_{\mathfrak L_{F_1}(k)}\,\le\, \|Q_{F_2}f\|^2_{\mathfrak L_{F_2}(k)}.
  \]  

  STEP 2: {\sl Any Cauchy sequence in the norm \eqref{norm?} converges pointwise:} To see this we take $F$ to be a singleton.\\
  
  STEP 3: {\sl Every Cauchy sequence converges to a function in $\mathfrak H$:} Indeed, from \eqref{norm?}, convergence in the norm implies convergence on finite sets, and the limiting function is
  obtained in the previous step. It remains to show that the function so obtained belongs to the space \eqref{space???}. This follows from the definition of \eqref{norm?} since the operators $Q_F$
  are defined on all functions, and we use the fact that a Cauchy sequence is uniformly bounded in the \eqref{norm?} norm and so gives a uniform bound on the $\|Q_Ff\|_{\mathfrak L_F(k)}$.\\
  
  STEP 4: {\sl The norm is defined by an inner product:} Since the norm is defined by an increasing limit, one can take the polarization formula for every $F$ and takes limit to obtain the
  corresponding inner product.\\

  STEP 5: {\sl The reproducing kernel property holds:} Let $s_0\in S$ and let $F=\left\{s_0,s_1,\ldots , s_n\right\}\subset S$, and let $f\in\mathfrak H$. We have
  \[
    \begin{split}
      \begin{pmatrix}k(s_0,s_1)  &k(s_0,s_1)&\cdots &k(s_0,s_n)\end{pmatrix}Q_F^{-1}\begin{pmatrix}f(s_0)\\ f(s_1)\\ \vdots \\ f(s_n)\end{pmatrix}
      &=\begin{pmatrix}1  &0&\cdots &0\end{pmatrix}\begin{pmatrix}f(s_0)\\ f(s_1)\\ \vdots \\ f(s_n)\end{pmatrix},\\
      &=f(s_0).
    \end{split}
  \]

To conclude one takes the limit for those $F\subset S$ which contain $s_0$.\\

\end{proof}

As a transition to the second construction, we note the following.
Denote by $\mathfrak M(k)$ the linear span of the delta measures $\delta_s$, and let $\mathfrak M_F(k)$ denote the linear span of the $\delta_s$ for $s\in F$.
The map $T_k$ defined by
\begin{equation}
\label{tkf}
T_k\left(\sum_{t\in F}z_t\delta_t\right)=\sum_{t\in F} z_t k_t
\end{equation}
is unitary from $\mathfrak M_F(k)$ onto $\mathfrak L_F(k)$, with matrix representation $k(t,s)\big|_{F\times F}$. \\

The map $T_k$ extends to a unitary map between the projective limits $\mathfrak D_1(k)=\lim_{\substack{F\rightarrow S \\ \mathcal F}}\mathfrak M_F(k)$ onto
$\mathfrak H(k)=\lim_{\substack{F\rightarrow S\\ \mathcal F}}\mathfrak L_F(k)$.

\begin{remark}
  {\rm Rather than restricting to a subset $S_0$ one can use the Moore-Penrose inverse $K_F^{[-1]}$ of the matrix $K_F$; see \cite{MR1987382} for the latter. Then, one needs
    to use Schur complement formulas. See  \cite{MR2061575,HornJohnson} for the latter.}
\end{remark}  
\section{Dual norms and Hilbert space construction}
\setcounter{equation}{0}
\label{dual123456789}
Recall that $\mathcal B_k$ is the Borel sigma-algebra generated by the metric $d_k$ defined in Section \ref{dist}.
Let $\xi$ be a signed measure on $\mathcal B_k$. It follows from the definition of the latter and the positivity of $k(t,s)$ that the integral
\begin{equation}
\label{xi-equiv}
\langle \xi\,,\,\xi\rangle_k=\iint_{S\times S}\overline{d\xi(t)}k(t,s)d\xi(s)
\end{equation}
is non-negative, but possibly infinite. The order of integration in \eqref{xi-equiv} is not important since $k(t,s)$ is continuous with respect to the underlying topology (see
Proposition \ref{procont}).
For two measures $\xi$ and $\eta$ for which \eqref{xi-equiv} is finite we set
\begin{equation}
\label{xi-equiv-123432`}
\langle \xi\,,\,\eta\rangle_k=\iint_{S\times S}\overline{d\eta(t)}k(t,s)d\xi(s).
\end{equation}
Since \eqref{xi-equiv} is a positive
(possibly degenerate) Hermitian form, the Cauchy-Schwarz inequality holds
\begin{equation}
\label{cs-098}  
|\langle \xi\,,\,\eta\rangle_k|^2\le \langle \xi\,,\,\xi\rangle_k \langle \eta\,,\,\eta\rangle_k
\end{equation}  
for any pair of signed matrices for which \eqref{xi-equiv} is finite.\smallskip

We will say that two signed measures for which \eqref{xi-equiv} is finite are equivalent if
\begin{equation}
\label{measures-equiv}  
\langle\xi_1-\xi_2\,,\,\xi_1-\xi_2\rangle_k=0.
\end{equation}
Thanks to \eqref{cs-098} we have an equivalence relation, which we denote by $\sim_k$. We use the same symbol for the equivalence class and for an element in the class, and
still denote by $\langle\cdot,\cdot\rangle_k$ the corresponding Hermitian form.

\begin{definition}
We denote by $\mathfrak M_1(k)$ the set of equivalence classes of signed measures on $\mathcal B_k$ for which \eqref{xi-equiv} is finite, and set
\begin{equation}
\label{equiv}  
\mathfrak B_1(k) =\left\{\xi\in \mathfrak M_1(k)\,:\,
\langle \xi,\xi\rangle_k\le 1\right\}.
\end{equation}
\label{m1}
\end{definition}

The space $\mathfrak M_1(k)$ endowed with the form \eqref{xi-equiv} is a pre-Hilbert space, and its dual (space of anti-linear continuous functionals)
$(\mathfrak M_1(k))^*$ is complete. We will consider the elements of the previously constructed space $\mathfrak H(k)$ as linear functionals on $\mathfrak M_1(k)$.
By the results of Section \ref{sec2} $(\mathfrak M_1(k))^*$ is a Hilbert space.

\begin{proposition}
The linear span of the delta measures $\delta_s$, $s\in S$ is dense in $\mathfrak M_1(k)$.
\end{proposition}

\begin{proof}
  Let $\mu$ be orthogonal to the measures $\delta_t$, $t\in S$.  Then,
  \[
    \iint_{S\times S} \delta_{t_0}(t)k(t,s)d\mu(s)=0,\quad \forall t_0\in S,
  \]
  i.e.
  \[
\int_Sk(t,s)d\mu(s)=0,\quad \forall t\in S.
\]
Hence
\[
  \langle \mu,\mu\rangle_k=  \iint_{S\times S}\overline{d\mu(t)}k(t,s)d\mu(s)=0
\]
and so $\mu=0$.
\end{proof}

The operator $T_k$ defined in \eqref{tkf} extends to $\mathfrak M_1(k)$:
\begin{proposition}
The map $T_k$ defined by
\begin{equation}
(T_k\mu)(t)=\int_Sk(t,s)d\mu(s)
\end{equation} 
is an isometry from $\mathfrak M_1(k)$ into $(\mathfrak M_1(k))^*$, which extends to a unitary map between $\mathfrak D_1(k)$ (the Hilbert space closure $\mathfrak M_1(k)$) and
$(\mathfrak M_1(k))^*$.
\end{proposition}

\begin{proof}
Let $\mu\in\mathfrak M_1(k)$. It defines an element in $(\mathfrak M_1(k))^*$ via the formula
\begin{equation}
(\widetilde{T_k\mu})(\xi)=\int_{S}\overline{d\xi(t)}{T_k(\mu)(t)}=\langle \mu\,,\, \xi\rangle
\end{equation}
which is continuous in the $\mathfrak M_1(k)$ topology thanks to \eqref{cs-098}.
\end{proof}

\begin{theorem}
\label{thm1}
Let $k(t,s)$ be positive definite on the set $S$, with associated sigma-algebra of Borel
sets $\mathcal B_k$ associated to $d_k$. Let $\mathfrak M_1(k)$ be the space of signed measures on $\mathcal B_k$ such that
\[
\iint_{S\times S}\overline{d\xi(t)}k(t,s)d\xi(s)<\infty.
\]  
The space of functions of the form $\widetilde{f}(\xi)=\langle f,\xi\rangle$ with $f\in(\mathfrak M_1(k))^*$ endowed with the norm
\begin{equation}
\label{norms}  
\|\widetilde{f}\|=\|f\|_{\mathfrak M_1(k))^*}
\end{equation}
(where $\|f\|_{\mathfrak M_1(k))^*}$ denotes the dual norm)
is a Hilbert space. The restrictions of its elements to the jump measures $\delta_s$ is the reproducing kernel Hilbert space with reproducing kernel $k$.
\end{theorem}

\section{Proof of Theorem \ref{thm1}}
\setcounter{equation}{0}
\label{sec1}

STEP 1: {\sl The dual $(\mathfrak M_1(k))^*$ is a Hilbert space}.\smallskip

This follows from Section \ref{sec2} but we give a specific argument here.
By a general result in functional analysis (see e.g. \cite{Brezis}), the dual is a Banach space. We show that its norm is defined by an inner product. To that purpose, let $f\in
(\mathfrak M_1(k))^*$. It extends to the completion $\mathfrak D_1(k)$ of $\mathfrak M_1(k)$. Since $\mathfrak M_1(k)$ is a pre-Hilbert space, then $\mathfrak D_1(k)$ is a
Hilbert space and so, by Riesz theorem applied to $\mathfrak D_1(k)$  there is an operator $U_k$ from $(\mathfrak M_1(k))^*$ into
$\mathfrak D_1(k)$ such that
\begin{equation}
\langle f , \xi\rangle=\langle U_kf\, ,\, \xi\rangle_{\mathfrak D_1(k)},
\end{equation}
where the brackets on the left denote the duality between $\mathfrak M_1(k)$ and its anti-dual. Setting $f=T_k(\mu)$ in the above equality we obtain
\[
\langle U_kT_k(\mu)\, ,\, \xi\rangle=\int_S\overline{(T_k\mu)(s)}d\xi(s)=\langle \mu, \xi\rangle_{\mathfrak D_1(k)}
\]
and so $U_kT_k=I$.\\

We define an inner product on $(\mathfrak M_1(k))^*$ as follows:
\[
\langle f, g\rangle=\langle U_kf\,,\, U_kg\rangle_{\mathfrak D_1(k)}.
\]  

For $f=T_k\xi$ and $g=T_k\eta$ two elements of $(\mathfrak M_1(k))^*$ we have
\[
\langle T_k\xi\,,\, T_k\eta\rangle_{(\mathfrak M_1(k))^*}=\langle U_kT_k\xi\,,\, U_kT_k\eta\rangle_{\mathfrak D_1(k)}=\langle\xi\, ,\, \eta\rangle_{\mathfrak D_1(k)}.
\]
For $\xi=\eta$ we have
\begin{equation}
  \langle  T_k\xi\, ,\, T_k\xi\rangle_{(\mathfrak M_1(k))^*}  =\|\xi\|_{\mathfrak D_1(k)}^2.
\end{equation}

STEP 2: {\sl The Banach norm in $(\mathfrak M_1(k))^*$ coincide with the Hilbert space norm.}\\

It follows from the fact that the Hilbert norm is the supremum of the inner products on the unit ball of $\mathfrak M_1(k)$.
But, by the definition of the norm as a supremum we have
\[
\begin{split}
\|T_k\xi\|^2_{(\mathfrak M_1(k))^*}&=\sup_{\eta\in\mathfrak B_1(k)}|(T_k\xi)(\eta)|^2\\
&=\sup_{\eta\in\mathfrak B_1(k)}|\langle\eta,\xi\rangle_k|^2\\
&\le \langle\eta,\eta\rangle_k\langle\xi,\xi\rangle_k\end{split}
\]
so that $\|T_k\xi\|^2\le \langle\xi,\xi\rangle_k$. Taking $\eta=\xi/\sqrt{\langle\xi,\xi\rangle_k}$ leads to equality.\\

STEP 3: {\sl $(\mathfrak M_1(k))^*$ is the reproducing kernel Hilbert space with reproducing kernel
\begin{equation}
\langle T_k\delta_t\, ,\, T_k\delta_s\rangle=k(t,s).
\end{equation}    
}

Let $f\in(\mathfrak M_1(k))^*$ and $s\in S$. We have:
\[
\langle f\,,\, T_k\delta_s\rangle=f(s)=\widetilde{f}(\delta_s).
\]

\section{An example}
\setcounter{equation}{0}
Every positive definite function is the covariance function of a centered Gaussian process; see e.g. \cite[pp. 38-39]{Neveu68}.
Applying this result to the inner product of a real Hilbert space $\mathfrak H$ with inner product $\langle\cdot,\cdot\rangle$ and associated norm
$\|\cdot\|$, we obtain a process called {\sl the associated Gaussian space}.
The elements of this space are the centered Gaussian variables $W_h$, $h\in\mathfrak H$, with law $N(0,\|h\|^2)$ and such that
\[
\mathbb E(W_{h_1}W_{h_2})=\langle h_1,h_2\rangle, \quad h_1,h_2\in\mathfrak H.
\]
Let
\begin{equation}
\label{k2}  
k(t,s)=\sum_{n=0}^\infty\frac{t^ns^n}{(n!)^2},
\end{equation}
This gives an isometry sending a measure into the corresponding sequence of moments in the weighted $\ell_2$ space.
Then
\[
d\mu\in\mathfrak M_1(k)\quad\iff\quad\sum_{n=0}^\infty\frac{\big|\int_{\mathbb R}  t^nd\mu(t)\big|^2}{(n!)^2}<\infty
\]
(condition on the moments of the measure).
The Gaussian measures
\[
d\mu_h(t)=\frac{1}{\sqrt{2\pi}\|h\|}e^{-\frac{t^2}{2\|h\|^2}}dt
\]
belong to $\mathfrak M_1(k)$. There are explicit formulas for the moments.
We recall in particular that
\begin{equation}
\mathbb E(W_{h}^{2n})=\frac{(2n)!}{2^nn!}\|h\|^{2n},\quad n=0,1,2,\ldots
\label{moment567}
\end{equation}
\begin{proposition}
In the above notation,
\begin{equation}
\langle \mu_{h_1},\mu_{h_2}\rangle_k=k\left(\frac{\|h_1\|_2}{2},\frac{\|h_2\|_2}{2}\right),\quad h_1,h_2\in\mathfrak H.
\end{equation}  
\end{proposition}

\begin{proof}
Using \eqref{moment567} we can write for $h_1,h_2\in\mathfrak H$:
\[
\begin{split}
\langle \mu_{h_1},\mu_{h_2}\rangle_k&=\sum_{n=0}^\infty\frac{\mathbb E(W_{h_1}^{2n})\mathbb E(W_{h_2}^{2n})}{((2n)!)^2}\\
&=\sum_{n=0}^\infty\frac{1}{(n!)^2}\left(\frac{\|h_1\|^2}{2}\right)^n\left(\frac{\|h_2\|^2}{2}\right)^n\\
&=k\left(\frac{\|h_1\|_2}{2},\frac{\|h_2\|_2}{2}\right).
\end{split}
\]
\end{proof}
We note that this last kernel was studied in \cite{daf1,MR3816055}.

\section{Measurable kernels: Dual norm construction}
\setcounter{equation}{0}
In the previous we considered a positive definite function $k(t,s)$ defined for $(t,s)\in S\times S$ and associated to it in a natural way a metric and a corresponding sigma-algebra.
The arguments still work when $S$ is endowed ahead of time with a sigma-algebra, say $\mathcal B$, and $k(t,s)$ is assumed only jointly measurable in $(t,s)$ with
respect to $\mathcal B$, and has the property that the double integral \eqref{xi-equiv}
\[
\langle \xi\,,\,\xi\rangle_k=\iint_{S\times S} \overline{d\xi(t)}k(t,s)d\xi(s)
\]
is positive or possibly infinite for every signed measure on $\mathcal B$. As in Section \ref{dual123456789} we define two signed measures $\xi_1$ and $\xi_2$ for which the above integral is finite, to be equivalent if
\eqref{measures-equiv} holds, and the space ${\mathfrak M_1}(k)$ is defined as in Definition \ref{m1}.

\begin{definition}
With the above notation, the function $k(t,s)$ is called positive definite.
\end{definition}  

For $\xi\in\mathfrak M_1(k)$ the operator $T_k$ defined by
\[
(T_k\xi)(t)=\int_Sk(t,s)d\xi(s)
\]  
defines an element of $(\mathfrak M_1(k))^*$  via
\[
(\widetilde{T_k\xi})(\mu)=\langle \xi,\mu\rangle_k
\]  
Theorem \ref{thm1} takes now the following form:

\begin{theorem}
\label{thm2}
Let $(S,\mathcal B)$ be a measurable space, and let $k(t,s)$ jointly measurable for $t,s\in S$ and positive definite on $S$.  
The space of functions of the form $\widetilde{f}(\xi)=\langle f,\xi\rangle$ with $f\in(\mathfrak M_1(k))^*$ endowed with the norm
\[
\|\widetilde{f}\|=\|f\|_{(\mathfrak M_1(k))^*}.
\]
(where $\|f\|_{(\mathfrak M_1(k))^*}.$ is the dual norm)
is a Hilbert space of functions defined on $\mathfrak M_1(k)$, with the reproducing kernel $\langle T_k\mu,T_k\nu\rangle$ and reproducing kernel property
\begin{equation}
\widetilde{f}(\xi)=\langle f, T_k\xi\rangle_{(\mathfrak M_1(k))^*}.
\end{equation}  
\end{theorem}

\begin{remark}{\rm
The framework considered here includes for instance the singular integral
kernel appearing in the study of the fractional Brownian motion, namely
\[
  k(t,s)=|t-s|^{2H-2}
\]
with Hurst constant $H\in[1/2,1)$; see \cite{MR1801485}. For Lebesgue measurable functions $f$ and associated signed measure $f(t)dt$ the condition
$\langle f(t)dt, f(t)dt\rangle_k<\infty$ can be rewritten as
\[
  \int_{\mathbb R}|\widehat{f}(u)|^2|u|^{1-2H}du<\infty,
\]
as follows from Parseval's equality and since the distributional Fourier transform of $|t|^{2H-2}$ is $c_H|u|^{1-2H}$ with $c_H=2\Gamma(2H-1)\sin((1-H)\pi)$.
See e.g. \cite[p. 170]{GS}.
For more general information on such integral operators, see  \cite{MR0290095}.}
\end{remark}
\section{Quadratic forms, operator ranges and reproducing kernels}
\setcounter{equation}{0}
Connections between operator ranges and reproducing kernel spaces are well known, but usually
considered in the setting of bounded operators. We here consider a more general case, using the theory of quadratic forms.
In the notation of the previous section we still assume that $S$ is a measure space, with $\sigma$-algebra $\mathcal B$ and we fix a positive $\sigma$-finite measure $\mu$ on $S$.
We define
\begin{equation}
\mathcal F_\mu=\left\{A\in\mathcal B\,;\, \mu(A)<\infty\right\},
\end{equation}
assume that $k(t,s)$ is jointly measurable and positive in the sense that
\begin{equation}
\label{bmu}  
Q(\varphi)=\iint_{S\times S}\overline{\varphi(t)}k(t,s)\varphi(s)(d\mu\times d\mu)(t,s)\ge 0
\end{equation}
for every $\varphi$ in the linear span $\mathfrak D$ of the characteristic functions $1_A$, with $A\in\mathcal F_\mu$.\smallskip 

Assuming $Q$ densely defined and closed, Kato's theorem on quadratic forms, see \cite{MR1335452}, ensures that there is a positive self-adjoint operator $T$ such that
\[
Q(\varphi)=\langle \varphi,T\varphi\rangle_{\mathbf L_2(S,d\mu)}=\|T^{1/2}\varphi\|^2_{\mathbf L_2(S,d\mu)}, \quad \varphi\in\mathfrak D.
\]
We then defined the associated reproducing kernel Hilbert space to be the range of $T^{1/2}$ with the operator range norm, meaning
\begin{equation}
\langle T^{1/2}\varphi,T^{1/2}\psi\rangle_T=\langle \varphi,(I-\pi)\psi\rangle_{\mathbf L_2(S,d\mu)},
\end{equation}  
where $\pi$ is the orthogonal projection onto the kernel (null space) of $T$. We then have 
\begin{eqnarray}
  \langle T\varphi,T\psi\rangle_T&=&\langle T\varphi,\psi\rangle_{\mathbf L_2(S,d\mu)},\\
  \langle T^{1/2}\varphi,T1_A\rangle_T&=&\langle T\varphi,1_A\rangle_{\mathbf L_2(S,d\mu)}=\int_A(T\varphi)(t)d\mu(t)
\end{eqnarray}
and the reproducing kernel is now
\[
k(A,B)=Q(1_{A\cap B})=\iint_{A\times B}d\mu(t)d\mu(s).
\]  

Within the present setting one can also construct the above reproducing kernel space as a projective limit.
We define
\begin{equation}
\mathfrak H_A(k)=\left\{ F(t)=\int_Ak(t,s)\varphi(s)d\mu(s),\,\,\varphi\in\mathfrak D\right\}
\end{equation}  
with inner product (with $G(t)=\int_Ak(t,s)\psi(s)d\mu(s)$, where $\psi\in\mathfrak D$)
\begin{equation}
\langle F,G\rangle=\iint_{A\times A}\overline{\psi(t)}k(t,s)\varphi(s)(d\mu\times d\mu)(t,s)
\end{equation}

We define the injective limit
\begin{equation}
\mathfrak H  =\lim_{\substack{A \rightarrow S\\ A\in\mathcal F_\mu}}\mathfrak H_A(k)
\end{equation}
and the following arguments are similar to the case of finite space considered earlier.
\begin{proposition}
Let $A\in\mathcal F_\mu$.  The map $T_A$ from $\mathbf L_2(A,\mu)$ onto $\mathfrak H_A$ defined by
\begin{equation}
T_A(\varphi d\mu)=\int_Ak(t,s)\varphi(s)d\mu(s)
\end{equation}  
is an isometry, and $T_AT_A^*=Q_A$ is the projection from $\mathbf L_2(S,\mu)$ onto $\mathbf L_2(A,\mu)$:
\begin{equation}
(Q_Ag)(t)=\int_Ak(t,s)T_A^*(g|_A)(s)d\mu(s)
\end{equation}  
\end{proposition}


The system of projections $\left\{Q_A,A\in\mathcal F_\mu\right\}$ is the counterpart of the system $\left\{Q_F,F\in\mathcal F\right\}$ in Section \ref{secF}, and
when $A\subset B$ we have $Q_A=Q_AQ_B$.\smallskip

%

\section{An application to stochastic processes}
\setcounter{equation}{0}
We now give an application to random variables of the construction presented in the previous section. More precisely, we start as before with $k(t,s)$ positive
definite on $S$, and we associate to $k(t,s)$ a measurable positive kernel defined in an underlying probability space. 
Let thus $(\Omega,\mathcal B,P)$ be a probability space. For every random variable $X$ with values in $S$, we set
\[
k_X(u,v)=k(X(u),X(v)),\quad u,v\in\Omega
\]  
to be the induced kernel on $\Omega\times \Omega$ defined from $X$.
For a complex-valued random variable $X$ on $\Omega$ we defined a signed measure on $S$ via
\begin{equation}
  W_X(P)=P\circ X^{-1}.
\end{equation}
We then have:
\begin{equation}
  \label{isometry789}
\begin{split}
\iint_{\Omega\times \Omega}k(X(u),X(v))dP(u)dP(v)
&=\iint_{S\times S}(d(\overline{P\circ  X^{-1})(t)})k(t,s)d(P\circ X^{-1})(s)
\end{split}
\end{equation}
when the second double integral is finite. We denote by $\mathfrak M_1(k_X)\subset\mathfrak M_1(k)$
the corresponding set of signed measures.
 We define a norm on the corresponding random variables ${\rm RV}_k(\Omega)$ via
\begin{equation}
\langle P, P\rangle=\langle W_PX,W_PX\rangle_{\mathfrak M_1(k)}
\end{equation}

From \eqref{isometry789} we have:

\begin{theorem}
  The map $W_X$ is an isometry from ${\rm RV}_k(P)$ into $\mathfrak M_1(k)$.
\end{theorem}

We have here the generalization of the notion of distribution of a random variable.\\

We give a Hilbert structure on a family of random variables on $\Omega$ using the previous analysis.

\[
k_{X,Y}(u,v)=k(X(u),Y(v)),\quad u,v\in\Omega
\]

\begin{equation}
  \label{isometry789}
\begin{split}
\iint_{\Omega\times \Omega}k(X(u),Y(v))dP(u)dP(v)
&=\iint_{S\times S}(d(\overline{P\circ  X^{-1})(t)})k(t,s)d(P\circ Y^{-1})(s)
\end{split}
\end{equation}

\section{Hausdorff distance}
\setcounter{equation}{0}
Finally we make some connections between our analysis and the Hausdorff distance of two measures in $\mathfrak M_1(k)$. Recall that the Hausdorff distance is defined to be
\begin{equation}
d_{\rm Haus}(\mu,\nu)=\sup_{f\in {\rm Lip}_1(S)}\left|\int_Sf(t)d\mu(t)-\int_Sf(t)d\nu(t)\right|
\end{equation}
where ${\rm Lip}_1$ denotes the set of functions on $S$ such that
\[
\|f\|_{\rm Lip}=\sup_{\substack{s,t\in S\\ s\not =t}}\frac{|f(t)-f(s)|}{d_k(t,s)}\le 1.
\]

\begin{proposition}
Let $\mu,\nu\in\mathfrak M_1(k)$. It holds that
\begin{equation}
d_{\rm Haus}(\mu,\nu)\le d_{\mathfrak H(k)}(\mu,\nu)=\|T_k(\mu)-T_k(\nu)\|_{\mathfrak H(k)}.
\end{equation}
\end{proposition}

\begin{proof}
By Cauchy-Schwarz inequality,
\[
\frac{|f(t)-f(s)|}{d_k(t,s)}=\frac{|\langle f,k_t-k_s\rangle_{\mathfrak H(k)}}{\|k_t-k_s\|_{\mathfrak H(k)}}\le \|f\|_{\mathfrak H(k)}.
\]
\end{proof}
\bibliographystyle{plain}

\begin{thebibliography}{10}

\bibitem{MR3546121}
R.A. Aliev and C.A. Gadjieva.
\newblock Approximation of hypersingular integral operators with {C}auchy
  kernel.
\newblock {\em Numer. Funct. Anal. Optim.}, 37(9):1055--1065, 2016.

\bibitem{daf1}
D.~{Alpay}, P.~{Jorgensen}, R.~{Seager}, and D.~{Volok}.
\newblock {On discrete analytic functions: Products, rational functions and
  reproducing kernels}.
\newblock {\em Journal of Applied Mathematics and Computing}, 41:393--426,
  2013.

\bibitem{MR3816055}
D.~Alpay and M.~Porat.
\newblock Generalized {F}ock spaces and the {S}tirling numbers.
\newblock {\em J. Math. Phys.}, 59(6):063509, 12, 2018.

\bibitem{aron}
N.~Aronszajn.
\newblock Theory of reproducing kernels.
\newblock {\em Trans. Amer. Math. Soc.}, 68:337--404, 1950.

\bibitem{MR1987382}
A.~Ben-Israel and T.N. Greville.
\newblock {\em Generalized inverses}, volume~15 of {\em CMS Books in
  Mathematics/Ouvrages de Math\'{e}matiques de la SMC}.
\newblock Springer-Verlag, New York, second edition, 2003.
\newblock Theory and applications.

\bibitem{MR3683690}
J.~Bouvrie and B.~Hamzi.
\newblock Kernel methods for the approximation of nonlinear systems.
\newblock {\em SIAM J. Control Optim.}, 55(4):2460--2492, 2017.

\bibitem{MR2061575}
S.~Boyd and L.~Vandenberghe.
\newblock {\em Convex optimization}.
\newblock Cambridge University Press, Cambridge, 2004.

\bibitem{Brezis}
H.~Brezis.
\newblock {\em Analyse fonctionnelle}.
\newblock {Masson, Paris}, 1987.

\bibitem{MR1801485}
T.E. Duncan.
\newblock Some applications of fractional {B}rownian motion to linear systems.
\newblock In {\em System theory: modeling, analysis and control (Cambridge, MA,
  1999)}, volume 518 of {\em Kluwer Internat. Ser. Engrg. Comput. Sci.}, pages
  97--105. Kluwer Acad. Publ., Boston, MA, 2000.

\bibitem{GS}
I.M.~Guel\cprime fand and G.E. Shilov.
\newblock {\em {Les distributions. Tome 1}}.
\newblock Collection Universitaire de Math\'ematiques, No. 8. Dunod, Paris,
  1972.
\newblock Nouveau tirage.

\bibitem{HornJohnson}
R.A. Horn and C.R. Johnson.
\newblock {\em Topics in matrix analysis}.
\newblock Cambridge University Press, Cambridge, 1994.
\newblock Corrected reprint of the 1991 original.

\bibitem{MR1335452}
Tosio Kato.
\newblock {\em Perturbation theory for linear operators}.
\newblock Classics in Mathematics. Springer-Verlag, Berlin, 1995.
\newblock Reprint of the 1980 edition.

\bibitem{meschkowski}
H.~Meschkovski.
\newblock {\em Hilbertsche {R\"aume} mit {K}ernfunktion}.
\newblock Springer--{V}erlag, 1962.

\bibitem{Neveu68}
J.~Neveu.
\newblock {\em Processus al{\'e}atoires gaussiens}.
\newblock Number~34 in S{\'e}minaires de math{\'e}matiques sup{\'e}rieures. Les
  presses de l'universit{\'e} de Montr{\'e}al, 1968.

\bibitem{MR1254059}
P.~Saint-Pierre.
\newblock Approximation of the viability kernel.
\newblock {\em Appl. Math. Optim.}, 29(2):187--209, 1994.

\bibitem{saitoh}
S.~Saitoh.
\newblock {\em Theory of reproducing kernels and its applications}, volume 189.
\newblock Longman scientific and technical, 1988.

\bibitem{MR1637511}
A.~J. Smola and B.~Sch\"{o}lkopf.
\newblock On a kernel-based method for pattern recognition, regression,
  approximation and operator inversion.
\newblock {\em Algorithmica}, 22(1-2):211--231, 1998.

\bibitem{MR0290095}
E.M. Stein.
\newblock {\em Singular integrals and differentiability properties of
  functions}.
\newblock Princeton Mathematical Series, No. 30. Princeton University Press,
  Princeton, N.J., 1970.

\bibitem{MR4033493}
S.~N. Vasil\cprime~eva and Yu.~S. Kan.
\newblock Approximation of probabilistic constraints in stochastic programming
  problems using a probability measure kernel.
\newblock {\em Avtomat. i Telemekh.}, 80(11):93--107, 2019.

\bibitem{MR3940262}
M.~Yousefi, K.~van Heusden, I.M. Mitchell, and G.A. Dumont.
\newblock Model-invariant viability kernel approximation.
\newblock {\em Systems Control Lett.}, 127:13--18, 2019.

\end{thebibliography}
\def\cprime{$'$} \def\cprime{$'$} \def\cprime{$'$}
  \def\lfhook#1{\setbox0=\hbox{#1}{\ooalign{\hidewidth
  \lower1.5ex\hbox{'}\hidewidth\crcr\unhbox0}}} \def\cprime{$'$}
  \def\cprime{$'$} \def\cprime{$'$} \def\cprime{$'$} \def\cprime{$'$}
  \def\cprime{$'$}

\end{document}